\documentclass{article}
\usepackage{amsmath,amssymb,color,url,epstopdf,mathtools,amsthm}
\usepackage[mathscr]{euscript}
\usepackage{float}
\usepackage{graphicx}
\usepackage[left=1.5in,right=1.5in,top=1in,bottom=1in]{geometry}
\usepackage{titlesec}

\newtheorem{thm}{Theorem}
\newtheorem{lem}[thm]{Lemma}
\newtheorem{prop}[thm]{Proposition}

\newtheorem{question}{Question}
\newtheorem{example}{Example}
\newtheorem*{remark}{Remark}
\newtheorem*{claim}{Claim}
\newcommand{\R}{\mathbb R}

\include{skewlines.jpg}

\begin{document}

\title{Contact structures induced by skew fibrations of $\R^3$}
\author{Michael Harrison\footnote{Portions of this work were completed while the author was in residence at MSRI during the Fall 2018 semester and supported by NSF Grant DMS-1440140.}}

\maketitle

\begin{abstract}
A smooth fibration of $\R^3$ by oriented lines is given by a smooth unit vector field $V$ on $\R^3$, for which all of the integral curves are oriented lines.  Such a fibration is called skew if no two fibers are parallel, and it is called nondegenerate if $\nabla V$ vanishes only in the direction of $V$.  Nondegeneracy is a form of local skewness, though in fact any nondegenerate fibration is globally skew.  Nondegenerate and skew fibrations have each been recently studied, from both geometric and topological perspectives, in part due to their close relationship with great circle fibrations of $S^3$.

Any fibration of $\R^3$ by oriented lines induces a plane field on $\R^3$, obtained by taking the orthogonal plane to the unique line through each point.  We show that the plane field induced by any nondegenerate fibration is a tight contact structure.  For contactness we require a new characterization of nondegenerate fibrations, whereas the proof of tightness employs a recent result of Etnyre, Komendarczyk, and Massot on tightness in contact metric 3-manifolds.

We conclude with some examples which highlight relationships among great circle fibrations, nondegenerate fibrations, skew fibrations, and the contact structures associated to fibrations.
\end{abstract}

\section{Introduction and Statement of Results}

\subsection{Skew and nondegenerate fibrations}
\label{sec:intro} A fibration of $\R^3$ by oriented lines is called \emph{skew} if no two fibers are parallel.  A simple example of a skew fibration may be constructed as follows: choose an oriented line $\ell$ in $\R^3$, and surround $\ell$ with a family of hyperboloids which foliate $\R^3 - \ell$.  Each hyperboloid is ruled, so we may view it as a collection of oriented skew lines.  The collection of all ruling lines, together with $\ell$, give the skew fibration of $\R^3$ depicted below.

\begin{figure}[h!t]
\centerline{
\includegraphics[width=2.3in]{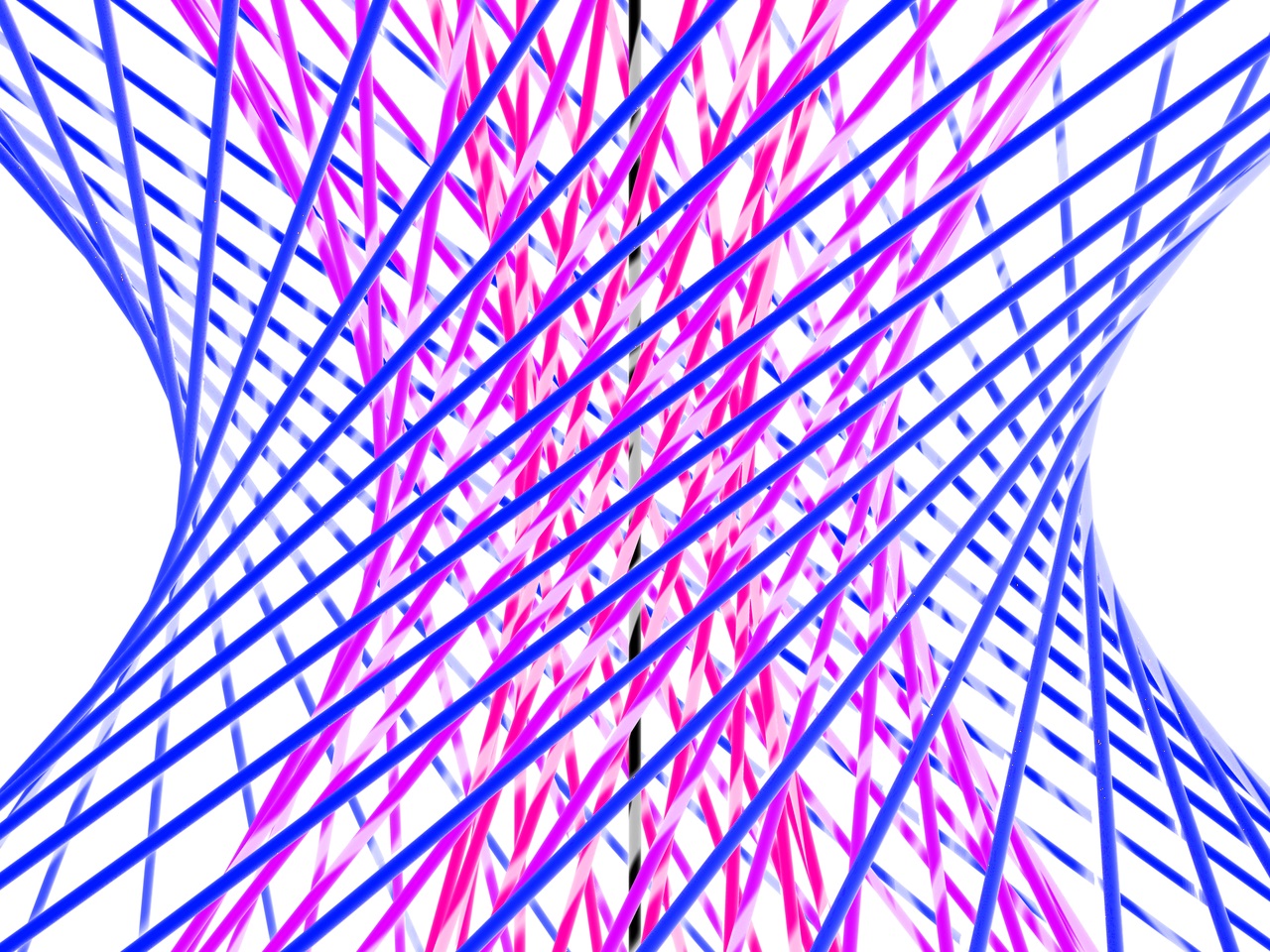}
}
\caption{A skew fibration of $\R^3$ (image by David Eppstein)}
\label{fig:hyper}
\end{figure}

Skew fibrations of $\R^3$ are intimately related to great circle fibrations of $S^3$.  Indeed, consider $S^3 \subset \R^4$, fibered by great circles.  Choose a tangent hyperplane $\R^3$ to $S^3$, say, at the north pole, and centrally project the open upper hemisphere of $S^3$ to $\R^3$.  The projection sends the nonintersecting great semicircles to nonintersecting lines.  Moreoever, since the great circles do not intersect on the equator $S^2$, the image lines do not meet at infinity -- that is, no two lines are parallel. Therefore, any great circle fibration projects to a skew fibration, and the space of all great circle fibrations sits naturally inside the space of skew fibrations.  The skew fibration depicted above is the image of the standard Hopf fibration of $S^3$.  In a recent paper \cite{Harrison}, the author showed that the space of all fibrations of $\R^3$ by skew, oriented lines deformation retracts to its subspace of Hopf fibrations, and furthermore, by restriction, the space of fibrations of $S^3$ by oriented great circles deformation retracts to its subspace of Hopf fibrations.  The latter statement was originally a theorem of Gluck and Warner \cite{GluckWarner}.

In 2009, prior to the author's topological treatment of skew fibrations, Salvai \cite{Salvai} gave a geometric classification (repeated here as Theorem \ref{thm:salvai}) of smooth fibrations of $\R^3$ by oriented lines.  A smooth fibration by oriented lines may be thought of as a unit vector field $V$ on $\R^3$, for which all of the integral curves are oriented lines.  This is equivalent to the condition that $\nabla V$ vanishes in the direction of $V$ at each point of $\R^3$.  Such a fibration is called \emph{nondegenerate} if $\nabla V$ vanishes \emph{only} in the direction of $V$.  Among smooth line fibrations, nondegenerate implies skew (\cite{Salvai}, Lemma 6), but the converse is not true, as demonstrated by Example \ref{ex:skew} in Section \ref{sec:examples} of this note.

\subsection{Contact structures on $\R^3$} Our present goal is to examine the relationship between nondegenerate fibrations and contact structures on $\R^3$.  A contact structure on $\R^3$ is a maximally non-integrable plane field $\xi$.  Any contact structure may be defined as the kernel of a $1$-form $\alpha$ with $\alpha \wedge d\alpha$ never zero.  A contact structure is called \emph{overtwisted} if there exists an embedded disk $D$ in $\R^3$ such that $\partial D$ is tangent to $\xi$ while $D$ is transverse to $\xi$ along $\partial D$.  Otherwise $\xi$ is called \emph{tight}.  A contact structure $\xi$ on $\R^3$ is \emph{tight at infinity} if it is tight outside a compact set.  The contact structures on $\R^3$ were classified by Eliashberg in \cite{Eliashberg}: up to isotopy, there is one tight contact structure, one overtwisted contact structure which is not tight at infinity, and a countable number of pairwise non-isotopic overtwisted contact structures which are tight at infinity.  These latter structures can be obtained by taking an overtwisted contact structure on $S^3$ and deleting a point (see \cite{Eliashberg}, Theorem 1.A for details), whereas the tight contact structure and the overtwisted-at-infinity contact structure can be found, respectively, in Examples \ref{ex:fibs} and \ref{ex:overtwisted} of this note.

\subsection{Main results} \label{sec:mainthm} For any fibration of $\R^3$ by oriented lines, there exists a canonical plane distribution: at each point the plane is the orthogonal complement of the unique fiber passing through that point.  Our main result is the following.

\begin{thm}
\label{thm:main}
The plane distribution induced by any smooth, nondegenerate fibration of $\R^3$ is a tight contact structure.
\end{thm}

At first glance, the contact assertion of Theorem \ref{thm:main} seems to be a purely local statement: the first-order contact condition follows from the first-order nondegeneracy condition.  However, consider the vector field $V(p) = \frac{p}{|p|}$ defined on $\R^3 - \left\{ 0 \right\}$.  This vector field induces a foliation of $\R^3 - \left\{ 0 \right\}$ by outward-pointing rays; moreover, the vector field is nondegenerate at every point of its domain.  But the induced plane distribution is not contact -- round spheres are tangent.  Therefore, the proof of Theorem \ref{thm:main} requires global considerations.

Regarding the tightness assertion of Theorem \ref{thm:main}, our argument applies to contact structures associated with a more general class of line fibrations.

\begin{thm}
\label{thm:main2}
Consider a fibration of $\R^3$ by oriented lines (not necessarily skew) such that the induced plane distribution $\xi$ is a contact structure.  If there exists any fiber which is parallel to no other fiber, then $\xi$ is tight. 
\end{thm}

In the final section we offer an example of a non-skew line fibration for which Theorem \ref{thm:main2} applies, resulting in a tight contact structure.

\begin{example}
\label{ex:fibs}
The following examples provide some context for the above theorems.  In particular, the plane field associated to a smooth, skew fibration may not be contact.  On the other hand, certain smooth non-skew fibrations induce contact structures. 
\begin{enumerate}
\item The plane field induced by the Hopf fibration depicted above is one of the usual representations of the standard (tight) contact structure on $\R^3$, defined in cylindrical coordinates $(r,\theta,z)$ by the $1$-form $dz + r^2 \ d\theta$.

\item The plane distribution induced by the smooth, skew, degenerate fibration of Example \ref{ex:skew} in Section \ref{sec:examples} is not a contact structure.

\item Consider the non-skew fibration of $\R^3$ defined as follows.  Let $(x,y,z)$ be rectangular coordinates on $\R^3$, and foliate $\R^3$ by planes $P_y$ parallel to the $xz$-plane.  Now fiber each individual plane by parallel lines: for the $xz$-plane $P_0$, use lines parallel to the $z$-axis; and for the plane $P_y$, use lines with direction $(-y,0,1)$.  Though non-skew, the induced plane distribution is a tight contact structure; in fact, it is another typical representation of the standard contact structure: $dz -y \ dx$.

\item Consider the planes $P_y$ as above, fiber the $xz$-plane $P_0$ by lines parallel to the $z$-axis, and again fiber each individual plane by parallel lines, but this time let the angle vary at a constant rate with respect to $y$.  This is yet another representation of the standard contact structure on $\R^3$.  This fibration is not skew but it is characterized (among fibrations of $\R^3$ by lines) by a property called \emph{fiberwise homogeneity}: for any two fibers $\ell_1$ and $\ell_2$, there is an isometry of Euclidean space which preserves the fibration and takes $\ell_1$ to $\ell_2$ (see \cite{Nuchi}).  The corresponding property on spheres characterizes the Hopf fibrations \cite{Nuchi2}.

\item To generalize the previous two examples, let $\alpha = \sin f(y) \ dx + \cos f(y) \ dz$ for some smooth $f : \R \to \R$.  Then $\ker \alpha$ is a (tight) contact structure as long as $f'(y) \neq 0$.

\end{enumerate}
These latter examples suggest the following question.
\end{example}

\begin{question} Does there exist a smooth line fibration (not necessarily skew) for which the induced plane distribution is a non-tight contact structure?
\end{question}

We believe the answer is \emph{no}.  In particular, if the hypothesis of Theorem \ref{thm:main2} does not apply, then every fiber admits parallel fibers.  Some preliminary investigations suggest that in this case, the fibration must be of the form of Example \ref{ex:fibs}(v): a foliation by parallel planes, which are individually foliated by lines.  To formalize this, we require some structural results for non-skew fibrations, which are a current topic of investigation by the author and Emmy Murphy.

\begin{remark} Following the submission of this article, Becker and Geiges posted a preprint \cite{BeckerGeiges} confirming the negative answer to the question above, by showing directly that when the hypothesis of Theorem \ref{thm:main2} does not apply, any contact structure induced by a line fibration is diffeomorphic to the standard one.
\end{remark}

Here we mention just one explicit example, which may demonstrate a certain incompatibility of overtwisted disks with line fibrations.

\begin{example} Consider the standard overtwisted contact structure $\xi_{ot}$, written in cylindrical coordinates as $\cos r \ dz + r \sin r \ d \theta$.  The map which sends a point $x \in \R^3$ to the direction of the oriented line orthogonal to the contact plane at $x$ is not a fibration.  In particular, considering $C = \left\{ r = \pi, z = 0 \right\}$, we see that all lines orthogonal to the contact structure on $C$ are vertical, and so these lines foliate the cylinder $r = \pi$.  Then any non-vertical line from the interior of the cylinder will pierce this cylinder.
\label{ex:overtwisted}
\end{example}
\subsection{Tightness of contact structures} In a recent preprint \cite{Gluck}, Gluck showed that great circle fibrations of $S^3$ induce tight contact structures on $S^3$.  The proof of tightness involves, given any contact structure $\xi$ induced by a great circle fibration, deforming $\xi$ to a Hopf contact structure and applying Gray stability.  In $\R^3$, though a similar deformation exists (see Theorem \ref{thm:nondegeneratedefret} below), we do not have Gray stability and must resort to other methods.

Since overtwisted disks can have very complicated geometry, it is often difficult to show that a contact structure is tight, even when the structure is given explicitly.  However, tightness in contact metric $3$-manifolds was recently and thoroughly studied by Etnyre, Komendarczyk, and Massot \cite{EtnyreEtal}.  It follows from \cite{EtnyreEtal}, Proposition 5.3 (repeated below as Proposition \ref{prop:tightness}), that if a contact structure orthogonal to a skew fibration is overtwisted, then there must exist an overtwisted disk for which the boundary lies on a round sphere centered at the origin.  So our proof of tightness involves showing that the skew condition prohibits the existence of a closed Legendrian curve on a round sphere.

In the specific case of nondegenerate fibrations, another possible tightness argument was suggested to us by an anonymous referee.  The idea uses the Complete Connection Criterion (\cite{EliashbergThurston}, Proposition 3.5.6), which we recount in Section \ref{sec:tightness}, together with the Continuity at Infinity property for skew fibrations (Lemma \ref{lem:contatinf}).  In Section \ref{sec:tightness} we offer a brief sketch of this possible proof.

\subsection{Higher-dimensional considerations}  We briefly mention the current status for higher-dimensional skew fibrations.  If $S^n \subset \R^{n+1}$ is fibered by great $p$-spheres, then the central projection from $S^n$ to any tangent $\R^n$ produces a fibration of $\R^n$ by \emph{skew} copies of $\R^p$ -- here two affine $p$-planes in $\R^n$ are called skew if they neither intersect nor contain parallel directions.  However, the condition on dimensions $p$ and $n$ such that skew fibrations exist is much less restrictive than in the spherical case.  In a 2013 paper \cite{OvsienkoTabachnikov}, Ovsienko and Tabachnikov showed that a fibration of $\R^n$ by skew oriented copies of $\R^p$ exists if and only if there exist $p$ linearly independent vector fields on $S^{n-p-1}$, which occurs, by a famous result of Adams \cite{Adams}, if and only if $p \leq \rho(n-p) - 1$, where $\rho$ is the classical Hurwitz-Radon function.  In particular, the only fibrations of $\R^{2p+1}$ by skew $\R^p$ occur when $p \in \left\{ 1, 3, 7 \right\}$.

Spherical fibrations have been extensively studied due, in part, to their relationship with the Blaschke conjecture (see \cite{McKay2} for a recent and thorough summary on the current progress of this conjecture).  On the contrary, skew fibrations have been studied only recently.  It seems that skew fibrations could be a useful tool in the study of spherical fibrations, in the dimensions where both objects exist.

%especially since we know that every fibration of $\R^{2p+1}$ by skew copies of $\R^p$ can be written globally and explicitly (see \cite{HarrisonThesis}).  In \cite{Harrison} and \cite{HarrisonThesis} we discuss the possibility of studying fibrations $S^3 \to S^7 \to S^4$ and $S^7 \to S^{15} \to S^8$ by instead studying their skew fibration counterparts.

Returning to the matter at hand, we wonder:

\begin{question}
Given a smooth, nondegenerate fibration of odd-dimensional Euclidean space by lines, is the induced hyperplane distribution contact?
\end{question}

The answer is yes for the fibrations of $\R^{2k+1}$ obtained by the projections of Hopf fibrations of $S^{2k+1}$ by great circles.

\begin{remark} Following the submission of this article, Gluck and Yang posted a preprint \cite{GluckYang} giving examples of great circle fibrations of $S^{2k+1}$, for every $k > 1$, which do not correspond to contact structures.  The same construction may be adapted to give a negative answer to the question above.  We still wonder: under what additional conditions on a line fibration is the induced hyperplane distribution contact?
\end{remark}

\textbf{Acknowledgments}. I am grateful to Yasha Eliashberg, John Etnyre, Herman Gluck, and Emmy Murphy for useful and stimulating discussions.  I would also like to thank the referee for comments and suggestions which greatly improved the presentation of this note.

\section{Background}
In the proof of the main result, we will make use of geometric and topological results from past studies of skew fibrations (\cite{Harrison}, \cite{Salvai}) and a strong tightness result for certain contact $3$-manifolds \cite{EtnyreEtal}. We collect these results here.

\subsection{Topology of skew fibrations}
\label{sec:topback} Let $F$ be a (continuous) fibration of $\R^3$ by skew, oriented lines, given by a vector field $V : \R^3 \to S^2$.  The fibration may be viewed locally as follows: given a point $o \in \R^3$, let $\ell$ be the oriented fiber through $o$ and let $\ell^\perp$ denote the affine subspace of $\R^3$ which passes through $o$ and is orthogonal to $\ell$.  There exists a neighborhood $E$ of $o$ in $\ell^\perp$ such that the fibers through points $p \in E$ are transverse to $\ell^\perp$.  In particular, fibers near $\ell$ may be written in terms of a map $B : E \to \R^2$, where $\R^2 \subset \R^3$ is a linear subspace parallel to $\ell^\perp$, and the fiber through $p$ is the graph of the affine map $\R \to \R^2 : t \mapsto p + tB(p)$.  That is, $B(p)$ gives the ``horizontal" (i.e.\ orthogonal to $\ell$) component of the direction of the fiber through $p$.  Outside of $E$, it is possible that fibers fail to be transverse to $\ell^\perp$, see fiber $\ell'$ in Figure \ref{fig:localfib} below.

\begin{figure}[h!t]
\centerline{
\includegraphics[width=3in]{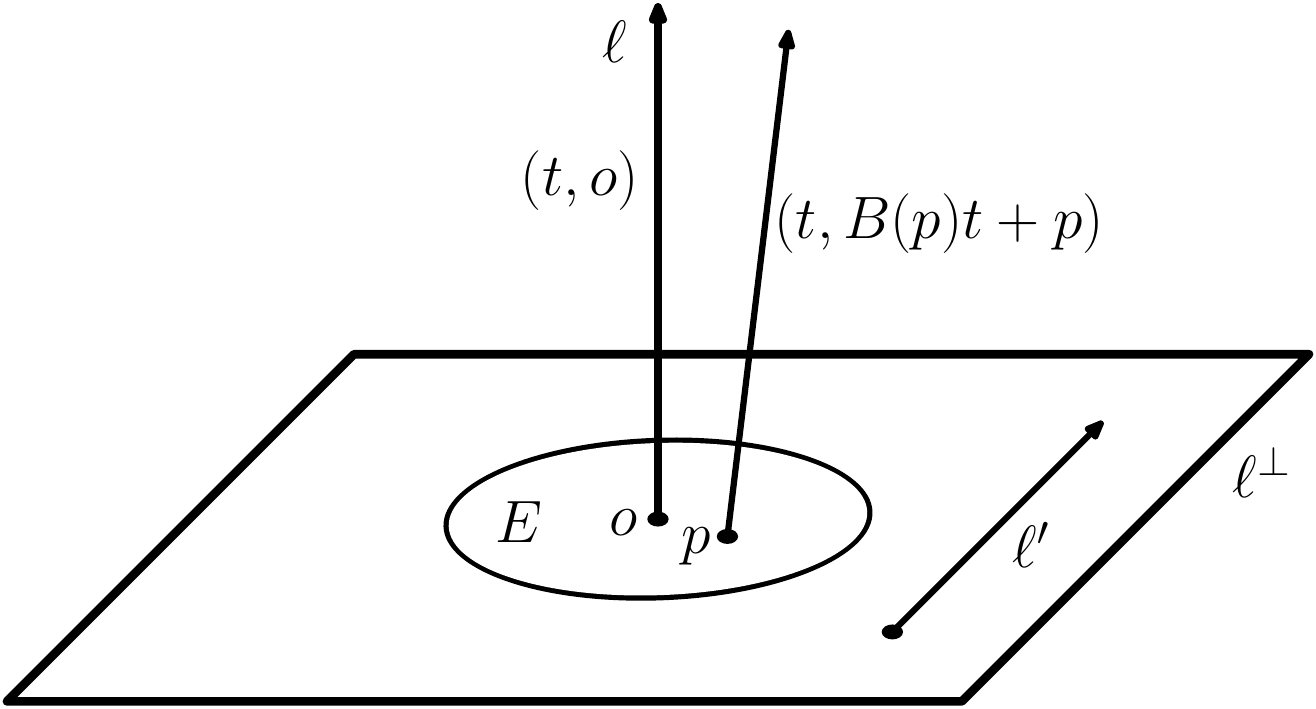}
}
\caption{Local depiction of a skew fibration}
\label{fig:localfib}
\end{figure}

When considering $B$ as a vector in $\R^3$, we use the notation ${\bf B}$.  This notation allows for an easy description in terms of the vector field $V$:
\[
{\bf B}(p) = \frac{1}{\langle V(p), V(o) \rangle}V(p) - V(o).
\]

Let $U \coloneqq V(\R^3)$ be the set of unit vectors which appear as directions of fibers from $F$.  It is not difficult to show that $U$ is an open, connected, contractible subset of $S^2$.  More importantly, $U$ is a convex subset of $S^2$ (see \cite{Harrison} and \cite{Salvai} for two different proofs of this fact).  In particular, for any skew fibration $F$, there exists a plane $P \subset \R^3$ which is transverse to all fibers from $F$.  Though the plane $P$ may not necessarily be unique, there is a canonical choice: since $U$ is convex, it has a unique circumcenter $u$; we choose $P$ to be the plane orthogonal to the fiber with direction $u$.

This choice is convenient because it varies continuously with the fibration.  In particular, the set of skew fibrations inherits a topology as a subset of the continuous maps $\R^3 \to S^2$. Then the image $U$ varies continuously with the fibration, and the circumcenter of a convex set varies continuously with the set (for a proof of this last statement, see e.g.\ \cite{GluckWarner}, Lemma 9.3).

Therefore, for any skew fibration, there corresponds a unique linear plane $P$, and the fibration may be written globally as a map $B : P \to \R^2$.

Given $p \in P$, let $d(p)$ be the (minimum) distance from the origin to the fiber through $p$.

\begin{lem}[(Harrison \cite{Harrison})]
\label{lem:harrison}
A continuous map $B : \R^2 \to \R^2$ corresponds to a skew fibration, in the manner described above, if and only if
\begin{itemize}
\item for any distinct points $p$ and $q$ in $\R^2$, $\det\big( \ p-q \ \ \ B(p)-B(q) \ \big) \neq 0$, and
\item if $p_n$ is a sequence of points in $\R^2$ with no accumulation points, then $d(p_n) \to \infty$.
\end{itemize}
Conversely, any skew fibration can be described by such a $B$ on a suitable $\R^2 \subset \R^3$. 
\end{lem}

The first item corresponds to skewness of the fibers through $p$ and $q$, and the second ensures that the collection of lines emanating from the domain $\R^2$ actually cover all of $\R^3$.

Given a skew fibration, consider the map $\R^2 \times \R^2 - \Delta(\R^2) \to \R$ which sends a pair of distinct points $(p,q)$ to $\det\left( \ p - q \ \ \ B(p) - B(q) \ \right)$.  Since the domain is connected and the image does not contain $0$, the image consists only of positive numbers or only of negative numbers.  Thus each fibration can be labeled as positively oriented or negatively oriented, according to the sign of the determinant.

\begin{example} Choose an oriented linear plane $P \simeq \R^2 \subset \R^3$.  Then the positively (resp. negatively) oriented Hopf fibration on $P$ is defined by $B(p) = ip$ (resp. $B(p) = -ip$).
\label{ex:hopf}
\end{example}

The following Continuity at Infinity property describes the behavior of skew fibrations at infinity.

\begin{lem}[(Harrison \cite{Harrison})]
\label{lem:contatinf}
Given a skew fibration $F$, let $U \subset S^2$ be the set of directions of fibers from $F$.  Let $u \in U$, let $p_n$ be a sequence of points in $\R^3$, and let $u_n \in U$ be the direction of the fiber through $p_n$.  If $|p_n| \to \infty$ and $\frac{p_n}{|p_n|} \to u$, then $u_n \to u$.
\end{lem}

For example, if $u$ is the vertical fiber in the Hopf fibration of Figure \ref{fig:hyper}, and if $p_n$ is a sequence of points which is ``eventually vertical" in the sense of Lemma \ref{lem:contatinf}, then the directions of the fibers through $p_n$ converge to $u$.

Continuity at Infinity does not necessarily hold for non-skew fibrations; see Example \ref{ex:fibs}(v).

\subsection{Geometry of nondegenerate fibrations}
\label{sec:salvai}
Here we discuss nondegenerate fibrations as studied in \cite{Salvai}.  An oriented affine line in $\R^3$ is an element of the oriented affine Grassmann manifold $\widetilde{AG}_1(3)$, which we will associate with $TS^2$.  Indeed, any oriented affine line $\ell$ may be uniquely described by the pair $(u,v)$, where $u \in S^2$ is the unit direction of $\ell$ and $v \perp u$ is the nearest point from $\ell$ to the origin.  So $\ell = (u,v) \in TS^2$.

There is a canonical pseudo-Riemannian metric of signature $(2,2)$ on $TS^2$ (studied in detail, as an indefinite K\"{a}hler metric on the space of oriented lines, by Guilfoyle and Klingenberg \cite{GuilfoyleKlingenberg}).  For $\zeta \in T_{(u,v)}TS^2$, define the square norm
\[
\mathscr{\bf Q} (\zeta) \coloneqq \langle \zeta_1 \times u, \zeta_2 \rangle = \det\big( \ \zeta_1 \ \ u \ \ \zeta_2 \ \big),
\]
where $\zeta_1, \zeta_2 \in u^\perp$ are the horizontal and vertical components of $\zeta$, $\times$ is the cross-product in $\R^3$, and $\langle \cdot, \cdot \rangle$ is the Euclidean inner product on $\R^3$.  We say that a surface $M$ in $TS^2$ is \emph{definite} if $\mathscr{\bf Q} \big|_M$ is definite; that is, for any $(u,v) \in M$ and any $\zeta \in T_{(u,v)}M$, $\mathscr{\bf Q}(\zeta) = 0 \Leftrightarrow \zeta = 0$.

\begin{thm}[(Salvai \cite{Salvai})]
\label{thm:salvai}
A surface $M \subset TS^2$ is the space of fibers of a nondegenerate smooth fibration of $\R^3$ by oriented lines if and only if $M$ is a closed (in the relative topology), definite, connected submanifold of $TS^2$.
\end{thm}

\subsection{Main technical lemma}  As discussed in Section \ref{sec:mainthm}, there can be no purely local proof that nondegeneracy implies contact.  To be precise, the nondegeneracy condition on $V$ is equivalent to the condition that $dB$ has full rank, but this is not enough to show that the induced plane distribution is contact (see the discussion following Theorem \ref{thm:main}).  However, one can use globality of the fibration to obtain the definiteness property of Theorem \ref{thm:salvai}, from which the contact condition will follow.  We present an intermediate lemma which incorporates this definiteness into Lemma \ref{lem:harrison}.  The statements below are proven in Section \ref{sec:proofs}.

Given $p \in \R^2$, recall that $d(p)$ refers to the (minimum) distance from the origin to the fiber through $p$.

\begin{lem}
\label{lem:noreal}
A smooth map $B : \R^2 \to \R^2$ corresponds to a nondegenerate smooth line fibration if and only if
\begin{itemize}
\item At any point $p \in \R^2$, the linear map $dB_p$ has no real eigenvalues.
\item if $p_n$ is a sequence of points in $\R^2$ with no accumulation points, then $d(p_n) \to \infty$.
\end{itemize}
Conversely, any nondegenerate fibration can be described by such a $B$ on a suitable $\R^2 \subset \R^3$. 
\end{lem}

Note that $dB_p$ has no real eigenvalues if and only if, for all nonzero vectors $h \in T_p\R^2 = \R^2$, $\det\left( \ h \ \ dB_p h \ \right) \neq 0$.  That is, the no real eigenvalue condition is equivalent to the condition that the quadratic form $Q : T_p\R^2 \to \R : h \mapsto \det\left( \ h \ \ dB_p h \ \right)$ is definite.  The sign of the definiteness is consistent for all $p \in \R^2$, so we obtain an orientation for any smooth nondegenerate fibration, which agrees with the orientation defined for skew fibrations after Lemma \ref{lem:harrison}.  The proof of Lemma \ref{lem:noreal} involves showing that the definiteness of $Q$ is equivalent to the definiteness of nondegenerate fibrations in Theorem \ref{thm:salvai}.

%The global characterization of nondegenerate fibrations in Lemma \ref{lem:noreal} is interesting in its own right, and it will be used for Theorem \ref{thm:nondegeneratedefret} and the examples in Section \ref{sec:examples}.  However, to prove the contact assertion of Theorem \ref{thm:main} we only need a local version of the no-real-eigenvalue property of Lemma \ref{lem:noreal}.  Given $p \in \R^3$, let $\ell$ be the unique fiber through $p$, let $u$ be the direction of $\ell$, and let $\ell^\perp$ be the plane perpendicular to $\ell$ and containing $p$.  Since $\nabla V$ vanishes in the direction of $V$, the nontrivial data of $\nabla V$ at $p$ is contained in the restriction $\nabla (V\big|_{\ell^\perp})(p) : T_p\ell^\perp \simeq \ell^\perp \to \ell^\perp \simeq T_uS^2$.

%\begin{lem}
%\label{lem:noreallocal}
%Let $F$ be a nondegenerate smooth line fibration of $\R^3$ and let $V$ be the corresponding unit vector field.  For any $p \in \R^3$, the linear map $\nabla (V\big|_{\ell^\perp})(p) : \ell^\perp \to \ell^\perp$ has no real eigenvalues.
%\end{lem}

We digress to state one other application of Lemma \ref{lem:noreal}: that the deformation retract of skew fibrations to Hopf fibrations, described in Section \ref{sec:intro}, preserves nondegeneracy in the following sense:

\begin{thm}
\label{thm:nondegeneratedefret}
The space of smooth nondegenerate fibrations of $\R^3$ deformation retracts to its subspace of Hopf fibrations.
\end{thm}

\subsection{Tightness in contact manifolds} \label{sec:tightness} We begin with definitions and notation from Section 5 of \cite{EtnyreEtal}.  For an oriented surface $S$ in a contact manifold $(\R^3, \xi)$, the characteristic foliation on $S$ can be obtained by integrating the singular line field given at $p \in S$ by $T_pS \cap \xi_p$.  A singularity of the characteristic foliation is a point $p$ for which $T_pS = \xi_p$.  Singularities of the characteristic foliation are positive or negative, depending on whether the orientation of $\xi_p$ matches or does not match that of $S$.  An oriented foliation on a sphere $S$ is \emph{simple} if it has exactly one singularity of each sign.  We will use the following result of Etnyre, Komendarczyk, and Massot, stated in the specific context of our setup in $\R^3$.

\begin{prop}[(Etnyre, Komendarczyk, Massot \cite{EtnyreEtal}, Proposition 5.3)]
\label{prop:tightness}
Let $B$ be a closed round ball of radius $r_0 > 0$, centered at the origin of the contact manifold $(\R^3, \xi)$.  If the characteristic foliation on every round sphere $S_r$, $0 < r \leq r_0$ is simple, and if $\xi|_B$ is overtwisted, then there is some radius $r_1$ such that the characteristic foliation on each $S_r$, $r \geq r_1$, has a closed leaf.
\end{prop}

Next, we recall the Complete Connection Criterion (\cite{EliashbergThurston}, Proposition 3.5.6), which we believe can be used to prove tightness in the nondegenerate case.  Suppose that there exists a plane $P = \R^2 \subset \R^3$ and a contact structure $\xi$ on $\R^3$, such that $\xi$ is everywhere transverse to the fibers of the orthogonal projection $\pi : \R^3 \to P$.  Then $\xi$ may be thought of as a connection for this trivial bundle $P \times \R$, though parallel translation is not necessarily globally defined.  The connection $\xi$ is called \emph{complete} if parallel translation is globally defined; that is, for any smooth (or rectifiable) path in $P$ and any lift of its initial point, there is an extension to a lift which is tangent to $\xi$ (i.e.\ a lift to a Legendrian curve).  Evidently, this will fail if the lift of some path blows up in finite time, which could happen if for some sequence of points $p_n$ in $\R^3$, which diverge in the vertical direction (orthogonal to $P$), the contact planes at $p_n$ become increasingly vertical.

\begin{prop}[(Complete Connection Criterion)] If there exists a plane $P$ in $(\R^3, \xi)$ such that $\xi$ is a complete connection for the fibration $\R^3 \to P$, then $\xi$ is tight.
\end{prop}

For a nondegenerate fibration, there exists a circumcenter $u \in U$ and a plane $P = u^\perp$ which is transverse to all fibers; equivalently, the contact planes are transverse to the fibers of the orthogonal projection $\pi : \R^3 \to P$.  We will call $u$ the vertical direction and $P$ horizontal.  The continuity at infinity of Lemma \ref{lem:contatinf} asserts that sequences of points which are ``eventually vertical" have fiber directions which converge to $u$.  Therefore, if $F \subset P$ is any compact set, then the set of directions $V(\pi^{-1}(F))$ is bounded away from horizontal; equivalently, the corresponding contact planes are sufficiently non-vertical.  Thus a finite time blow-up of a Legendrian lift should not occur.  To formalize this argument, one should apply a global existence and uniqueness result (\cite{Simmons}, Chap.\ 70, Theorem B) to the ODE governed by the Legendrian condition, with initial value given by the lift of the initial point.  The proof requires certain Lipschitz bounds and involves some technical estimates on derivatives, and we do not include the details here.

It is worth emphasizing the importance of the Continuity at Infinity condition.  There exist overtwisted contact structures which are transverse to all vertical lines (see the discussion following Proposition 3.5.6 of \cite{EliashbergThurston}), and so it is not sufficient to know that the fibers of a nondegenerate fibration are transverse to a common plane.

\section{Proofs}
\label{sec:proofs}

We begin with the proofs of the theorems, assuming the statement of Lemma \ref{lem:noreal}.

\begin{proof}[Proof of Theorem \ref{thm:main}]
Suppose $F$ is a nondegenerate smooth line fibration given by a smooth unit vector field $V$, and let $p \in \R^3$.  Let $\alpha$ be the $1$-form dual to $V$, and let $\ell^\perp$ denote the oriented plane containing $p$ and orthogonal to the vector $V(p)$.  As discussed at the beginning of Section \ref{sec:topback}, we may represent the fibration locally as a map $B : E \to \R^2$, where $E$ is a neighborhood of $p$ in $\ell^\perp$ such that the fibers through points of $E$ are transverse to $\ell^\perp$, and $\R^2 \subset \R^3$ is the linear subspace parallel to $\ell^\perp$.  For $B(p) \in \R^2$ and $h \in T_pE$, we use bold letters ${\bf B}(p)$ and ${\bf h}$ to represent the associated (by inclusion) vectors in $\R^3$.  This allows us to write, for $q \in E$, 
\[
{\bf B}(q) = \frac{1}{\langle V(q), V(p) \rangle}V(q) - V(p).
\]
We compute
\[
d{\bf B}_q h = \frac{1}{\langle V(q), V(p) \rangle}dV_q{\bf h} - \frac{\langle dV_q {\bf h}, V(p) \rangle}{\langle V(q), V(p) \rangle^2}V(q),
\]
where $h \in T_q E$.  Evaluating at $q = p$ yields
\[
d{\bf B}_p h = dV_p {\bf h},
\]
as $V$ is a unit vector field, and $dV_p {\bf h}$ is orthogonal to $V(p)$.  Recall, from the discussion following Lemma \ref{lem:noreal}, the quadratic form $Q(h) = \det\left( \ h \ \ dB_p h \ \right)$.  By the computation above, and by the fact that ${\bf h}$ and $d{\bf B}_p h$ are orthogonal to the oriented unit vector $V(p)$, we may write:
\[
Q(h) = \det\big( \ h \ \ dB_p h \ \big) = \det\big( \ {\bf h} \ \ dV_p{\bf h} \ \ V(p) \ \big) = \langle {\bf h} \times dV_p{\bf h}, V(p) \rangle,
\]
so that
\[
\operatorname{trace}(Q) = \langle \operatorname{curl}(V(p)), V(p) \rangle = \ast (\alpha(p) \wedge d\alpha(p)),
\]
where $\ast$ represents the Hodge star isomorphism.  By (the discussion following) Lemma \ref{lem:noreal}, $Q$ is definite, hence its trace is nonzero, and so $\alpha$ is a contact form. 

It remains to show that the contact structure is tight.  Because a nondegenerate fibration satisfies the hypothesis of Theorem \ref{thm:main2}, this is a consequence of the following proof.
\end{proof}

\begin{proof}[Proof of Theorem \ref{thm:main2}]  We will use the definitions and notation of Section \ref{sec:tightness}.

Consider $(\R^3, \xi)$, for which the contact structure $\xi$ is obtained from a line fibration, and let $\ell$ be a fiber which is parallel to no other fibers.  Choose a point on $\ell$ as the origin, and let $u$ be the oriented direction of $\ell$.  We first note that the characteristic foliation on every round sphere $S$ is simple.  Indeed, singularities occur at points $p \in S$ for which $T_pS = \xi_p$, or equivalently, points $p$ for which the fiber through $p$ passes through the origin.  So there are exactly two singular points of $S$, at the points where $\ell$ intersects $S$.  Therefore, by Proposition \ref{prop:tightness}, checking tightness amounts to checking that there exist no closed leaves of the characteristic foliation.  That is, we must check that no sphere contains a simple closed Legendrian curve.

Suppose that such a curve $C$ does exist, and note that $C$ is disjoint from the north and south poles.  Consider the height function, with respect to the direction $u$ of $\ell$ and restricted to $C$.  This function must have a critical point $q$, and the tangent line $T_qC$ is perpendicular to $u$.  Since this tangent line is contained in the contact plane $\xi_q$, the fiber through $q$ must be contained in the plane orthogonal to $T_qC$; in particular, this is the plane through the origin spanned by $u$ and the vector $q$.  But this plane also contains the line $\ell$, so the fiber through $q$ either intersects or is parallel to $\ell$, a contradiction.
\end{proof}

\begin{proof}[Proof of Theorem \ref{thm:nondegeneratedefret}]
We first recount the proof that the space of skew fibrations deformation retracts to its subspace of Hopf fibrations.  Given any skew fibration $F$, choose $P$ as described in Section \ref{sec:topback}: passing through the origin of $\R^3$ and orthogonal to the fiber $\ell$ with direction $u$, where $u$ is the circumcenter of the convex set $U$.  Parallel translate $F$ so that $P \cap \ell$ is the origin and represent this translated $F$ as a map $B$.  This translation is a path through skew fibrations which respects orientation.  Next define the straight-line homotopy $B_t(p) = (1-t)B(p) + tH(p)$, where $H$ is the unique Hopf fibration defined on $P$ whose orientation matches that of $B$ (see Example \ref{ex:hopf}).  It is shown in \cite{Harrison} that $B_t$ is a path through skew fibrations - in particular Lemma \ref{lem:harrison} holds for all $B_t$, $0 \leq t \leq 1$.  Since the plane $P$ (and therefore the Hopf fibration $H$ at the end of the path) was chosen to vary continuously with $F$, we obtain the result that the space of skew fibrations of $\R^3$ by oriented lines deformation retracts to its subset of Hopf fibrations.

It remains to show that the straight-line homotopy preserves nondegeneracy.  This is a consequence of the linearity of the map $\operatorname{Hom}(\R^2, \R^2) \to \R : L \mapsto \det(h,Lh)$ and the convexity of each connected component of $\R - \left\{ 0 \right\}$.  In particular,
\[
\det\big( \ h \ \ d(B_t)_p h \ \big) = (1-t)\det\big( \ h \ \ dB_ph \ \big) + t\det\big( \ h \ \ dH_p h \ \big) \neq 0,
\]
since each of the summands has the same sign.
\end{proof}

%We write
%\[
%\det(h \ \ d(B_t)_p h) = \det(h \ \ (1-t)dB_p h + tdH_p h) = (1-t)\det(h \ \ dB_p) + t\det(h \ \ dH_p h).
%\]
%Each of the summands has the same sign, so the determinant is nonzero for every $p$, $t$, and nonzero $h$.
%Note that for fixed $h$, the map $\Hom(\R^2, \R^2) \to \R : L \mapsto \det(h \ \ Lh)$ is linear.  Therefore, the image of the map $t \mapsto \det(h \ \ d(B_t)_p h)$ is a line segment, contained in the component of $\R - \left\{ 0 \right\}$ corresponding to the sign of the orientation of $B$.  Hence for every time $t$ and nonzero $h$, the determinant is nonzero.

Given the possibility to describe any nondegenerate fibration by this map $B : \R^2 \to \R^2$, we suspect it may be possible to explicitly write the contactomorphism taking the corresponding contact structure to the Hopf one.

\begin{proof}[Proof of Lemma \ref{lem:noreal}] Let $B : P \to P$ be a smooth map, where $P \subset \R^3$ is an oriented, linear $2$-plane.  Let $V_0$ represent the oriented unit vector in $\R^3$ which is orthogonal to $P$ (note that $V_0$ occurs as a fiber direction if and only if $0$ is in the image of $B$).  We will use the boldface letters ${\bf p}$, ${\bf B}$, and ${\bf h}$, to represent vectors in $\R^3$ corresponding via inclusion to vectors of $P$ or $T_pP$.  Let $M \subset TS^2$ be the collection of lines corresponding to $B$.  In particular, the fiber $\ell$ through a point $p \in P$ corresponds to the pair $(u,v) \in M \subset TS^2$, where $u \in S^2$ is the unit direction of $\ell$ and $v \in T_uS^2 \simeq \ell^\perp$ is the point on $\ell$ nearest to the origin.  Using the above notation, we may write $M$ as the image of the map $f = (u,v) : P \to M$ as follows:
\[
u(p) = \frac{{\bf B}(p) + V_0}{\sqrt{|{\bf B}(p)|^2+1}} \in S^2 \subset \R^3,
\hspace{.2in}
v(p) = {\bf p}
-\langle {\bf p}, u(p) \rangle
u(p) \in \ell^\perp.
\]

We must show that the two bullet points of Lemma \ref{lem:noreal} hold if and only if $M$ corresponds to a nondegenerate smooth fibration, which occurs if and only if $M$ is a closed, connected, definite surface in $TS^2$ (by Theorem \ref{thm:salvai}).  The connectedness of $M$ follows from the continuity of $B$, and it is straightforward to show that the second bullet point of Lemma \ref{lem:noreal} is equivalent to the fact that $M$ contains its limit points.

Thus it only remains to show that $M$ is definite if and only if $dB_p$ has no real eigenvalues for every $p \in P$, which occurs if and only if, for every $p \in P$, the quadratic form $Q : T_pP \to \R : h \mapsto \det( \ h \ \ dB_ph \ )$ is definite.  First observe that $Q(h)$ may be rewritten as $\det( \ {\bf h} \ \ d{\bf B}_ph \ \ V_0 \ )$, since $V_0$ is the oriented unit vector orthogonal to $P$.

Consider the pullback, with respect to the map $f = (u,v)$ onto $M$, of the quadratic form $\mathscr{\bf Q}$ on $TS^2$.  In particular, $f^*\mathscr{\bf Q}$ maps $h \in T_pP$ to $\mathscr{\bf Q}(df_p h) = \det( \ du_p h \ \ u(p) \ \ dv_p h \ )$.  The following claim asserts that definiteness of $\mathscr{\bf Q}|_M$ is equivalent to definiteness of $Q$.

\begin{claim}$Q(h) = \langle u(p), V_0 \rangle^2 f^*\mathscr{\bf Q}(h).$
\end{claim}

To compute $Q(h) = \det( \ {\bf h} \ \ d{\bf B}_ph \ \ V_0 \ )$, observe that as the first two vectors are each orthogonal to $V_0$, we may replace $V_0$ by $\frac{1}{\langle u(p), V_0 \rangle} u(p)$, since their components in the direction of $V_0$ are both equal $1$.  Now we write ${\bf B}(p) = \frac{1}{\langle u(p), V_0 \rangle}u(p) - V_0$, as in the proof of Theorem \ref{thm:main}.  We compute
\[
d{\bf B}_ph = \frac{1}{\langle u(p), V_0 \rangle} du_p(h) - \frac{\langle du_p(h), V_0 \rangle}{\langle u(p), V_0 \rangle^2}u(p),
\]
and note that the second term does not contribute to the determinant, due to our previous replacement of $V_0$.  Thus we have
\begin{align}
\label{eqn:q}
Q(h) = \det\big( \ {\bf h} \ \ d{\bf B}_ph \ \ V_0 \ \big) = \frac{1}{\langle u(p), V_0 \rangle ^2} \det\big( \ {\bf h} \ \ du_p(h) \ \ u(p) \ \big).
\end{align}
Let us also compute $f^*\mathscr{\bf Q}(h)$, beginning with
\[
dv_p h = {\bf h} - \langle {\bf h}, u(p) \rangle u(p) - \langle {\bf p}, du_p h \rangle u(p) -\langle {\bf p}, u(p) \rangle du_p h.
\]
The final three terms do not contribute to the determinant, hence:
\begin{align}
\label{eqn:Q}
f^*\mathscr{\bf Q}(h) = \det\big( \ du_p h \ \ u(p) \ \ dv_p h \ \big) = \det\big( \ du_p h \ \ u(p) \ \ {\bf h} \ \big).
\end{align}
Combining (\ref{eqn:q}) and (\ref{eqn:Q}) yields the desired result.
\end{proof}

\section{Examples}
\label{sec:examples}

We conclude with three examples.

\begin{example}[(A smooth, skew, degenerate fibration)]
\label{ex:skew}
Define $B : \R^2 \to \R^2$ as $B(p_1,p_2) = (-p_2^k,p_1^k)$ for any odd $k > 1$.  Then $B$ is a smooth map satisfying the hypotheses of Lemma \ref{lem:harrison}, so $B$ corresponds to a skew fibration.  However, the linear map $dB_{(0,0)}$ is the zero map, which has $0$ as a real eigenvalue.  Hence by Lemma \ref{lem:noreal}, $B$ corresponds to a degenerate fibration.  Finally, to show that the induced fibration $V : \R^3 \to S^2$ is smooth, we may define
\[
V(x,y,z) \coloneqq \frac{1}{1+p_1^{2k} + p_2^{2k}}\left( -p_2^k, p_1^k, 1 \right),
\]
where $p : \R^3 \to P$ is the map sending $(x,y,z)$ to the unique point at which $P$ intersects the fiber through $(x,y,z)$.  This map $p = (p_1,p_2)$ is defined implicitly by the equations
\[
x = p_1 - zp_2^k \hspace{.25in} y = p_2 + zp_1^k.
\]
We define $\Phi : \R^5 \to \R^2$ as $\Phi(x,y,z,p_1,p_2) = (p_1 - zp_2^k - x, p_2 + zp_1^k - y)$.  Then $\Phi$ is a smooth map, and we compute
\[
\left(\frac{\partial \Phi}{\partial p}\right) = \left( \begin{array}{cc}
1 & -kzp_2^{k-1} \\
kzp_1^{k-1} & 1
\end{array}
\right),
\]
which has nonzero determinant $1 + k^2z^2(p_1p_2)^{k-1}$ (since $k$ is odd).  Therefore by the Implicit Function Theorem, $p$ is a smooth function of $(x,y,z)$ and hence $V$ is smooth.

Note further that if we define the $1$-form $\alpha$ as in the proof of Theorem \ref{thm:main}, so that the plane field induced by $B$ is the kernel of $\alpha$, then $d\alpha(0) = 0$, so the distribution is not contact.  This example highlights the difference between the first-order nondegeneracy condition and the zero-order skew condition (compare with the family of functions $f_k : \R \to \R : x \mapsto x^k$, odd $k>1$; these are smooth maps which are topological embeddings but not immersions).
\end{example}

\begin{example}[(A nonskew fibration corresponding via Theorem \ref{thm:main2} to a tight contact structure)]
The example arises by gluing together the line fibration corresponding to the standard contact structure $dz - y \ dx$ with a line fibration corresponding to a Hopf-like contact structure.  In particular, we note that for both fibrations, the fiber through the point $(0,y,0)$ has direction $(-y,0,1)$.  As $y$ ranges, the collection of such fibers disconnects $\R^3$, and so we may cover one component with Hopf-like fibers and the other component with the fibers from the standard fibration.  In particular, we define
\[
B(p_1,p_2) = \left\{ \begin{array}{ccc} & \ (-p_2,p_1^3)  & \ \ p_1 \geq 0 \\ & (-p_2,0)  & \ \ p_1 < 0 \end{array} \right.
\]
The fibration is not skew, but it corresponds to a contact structure, tight by Theorem \ref{thm:main2}.
\end{example}

It is shown in \cite{Harrison} that a skew fibration of $\R^3$ corresponds via (inverse) central projection to a great circle fibration if and only if the map $B$ is surjective, which occurs if and only if the set of directions $U$ appearing in the fibration is an open hemisphere.  We have not seen in the literature any example of a skew fibration for which $U$ is not an open hemisphere, and so we provide an example.  The idea is a simple modifcation of the Hopf fibration $B(p_1,p_2) = (-p_2,p_1)$.  In the Hopf fibration, as $|p| \to \infty$, the fiber directions limit to horizontal.  Instead, we stop the direction from increasing past some specified angle.

\begin{example}[(A smooth nondegenerate fibration which does not correspond to a great circle fibration)] 
Choose any diffeomorphism $f : (-\varepsilon,\infty) \to (-\delta,a)$, such that $f(0) = 0$, $f'(x) > 0$ and $a > 0$ is finite; e.g.\ $f(x) = \arctan x$.  Let $B : \R^2 \to \R^2$ be given by $B(0) = 0$, $B(p) = \frac{f(|p|)}{|p|} ip$.  Then $B$ is not surjective since $|B(p)| < a$ for all $p$.  We may check that $B$ is a nondegenerate fibration by checking both items of Lemma \ref{lem:noreal}.  Since $B(p)$ is orthogonal to $p$, we have $d(p) = |p|$, so the second item is trivially satisfied.  To check the first item, we let $S = S(p_1,p_2) \coloneqq |p| = \sqrt{p_1^2 + p_2^2}$ and compute for $p \neq 0$:

\[
A \coloneqq \left(\frac{\partial B}{\partial p}\right) = \frac{1}{S^3} \left( \begin{array}{cc} p_1p_2(f(S) - Sf'(S)) & -p_2^2Sf'(S) - p_1^2f(S) \\ p_1^2Sf'(S) + p_2^2f(S) & p_1p_2(Sf'(S) - f(S)) \end{array} \right).
\] 
The eigenvalues of $A$ are the roots of the characteristic polynomial $\lambda^2 +$ tr$(A)\lambda + \det(A) = 0$.  The trace is zero, and we compute
\[
\det(A) = \frac{f(S)f'(S)}{S} > 0,
\]
so there are no real eigenvalues for $p \neq 0$.  It remains to check nondegeneracy at $p = 0$.  We claim that
\begin{align}
\label{eqn:deriv}
\frac{\partial B}{\partial p}(0) = \left( \begin{array}{cc} 0 & -f'(0) \\ f'(0) & 0 \end{array} \right)
\end{align}
and check by the definition of derivative:
\[
\lim_{|h| \to 0} \frac{B(h) 
- \left(\begin{array} {cc} 0 & -f'(0) \\ f'(0) & 0 \end{array} \right)  \left(\begin{array}{c} h_1 \\ h_2 \end{array} \right)}{|h|}
= \lim_{|h| \to 0} \left(\frac{f(|h|)}{|h|} - f'(0)\right)\frac{ih}{|h|}.
\]
The limit of the left factor is zero by definition of $f'(0)$, and each component of the vector on the right is bounded above by $1$, so the limit of the product exists and is $0$.  Hence the derivative is as claimed in (\ref{eqn:deriv}), and this has no real eigenvalues.  Thus the fibration is nondegenerate.
\end{example}

%%%%%BIBLIOGRAPHY
\bibliographystyle{plain}
\bibliography{bib}{}
%%%%%BIBLIOGRAPHY

\end{document}